\documentclass{amsart}
\usepackage[all]{xy}
\usepackage{amsmath,amsthm, rlepsf}
\usepackage{graphics}
\usepackage{color}
\usepackage{epsfig}

\usepackage{palatino} 

\usepackage{hyperref}

\theoremstyle{plain}
\newtheorem{thm}{Theorem}[section]

\newtheorem{prop}[thm]{Proposition}
\newtheorem{lem}[thm]{Lemma}

\theoremstyle{definition}
\newtheorem{remark}[thm]{Remark}

\newcommand{\R}{\ensuremath{\mathbb{R}}}
\newcommand{\Z}{\ensuremath{\mathbb{Z}}}
\newcommand{\C}{\ensuremath{\mathbb{C}}}
\def\co{\colon\thinspace}

\def\dfn#1{{\em #1}}

\title{Contact structures on 5--manifolds}
\author{John B. Etnyre}
\address{School of Mathematics \\ Georgia Institute of Technology}
\email{etnyre@math.gatech.edu}
\urladdr{\href{http://www.math.gatech.edu/~etnyre}{http://www.math.gatech.edu/\~{}etnyre}}

\begin{document}

\begin{abstract}
Using recent work on high dimensional Lutz twists and families of
Weinstein structures we show that any almost contact structure on a
5--manifold is homotopic to a contact structure.
\end{abstract}

\maketitle

\section{Introduction}

Though contact structures have been studied for quite some time we
still do not know which odd dimensional manifolds support such
structures. Recall that (oriented) contact structures only exist on
odd dimensional manifolds $M^{2n+1}$ and are described as $\xi=\ker
\alpha$ where $\alpha$ is a 1--form for which $\alpha\wedge
(d\alpha)^n$ is a volume from on $M$. (We will always take $M$
oriented and assume this volume form gives the preferred orientation.)
Noting that $d\alpha$ defines a symplectic form on $\xi$ we see that a
contact structure gives a reduction of the structure group of $TM$ to
$U(n)\times \mathbf{1}$ (more specifically choosing a complex
structure $J$ on $\xi$ that is compatible with $d\alpha$ gives the
reduction and one may easily check this reduction is independent of
the choice of $J$ or $\alpha$).  Any reduction of the structure group
of $TM$ to $U(n)\times \mathbf{1}$, or equivalently a choice of
hyperplane field $\eta$ in $TM$ with a complex structure $J$ on
$\eta$, is called an \dfn{almost contact structure} and the
fundamental existence question in contact geometry concerns whether or
not almost contact structures always come from contact structures.

For open manifolds Gromov \cite{Gromov69} showed that all almost
contact structures are homotopic to contact structures, but on closed
manifolds much less is known. Previously we only had a complete answer
in dimensions 1 and 3. The 1 dimensional result is trivial. In
dimension 3 almost contact structures are equivalent to plane fields
and so all oriented 3--manifolds $M$ admit almost contact
structures. Martinet \cite{Martinet71} showed that all closed oriented
3--manifolds admit a contact structure and Lutz \cite{Lutz71} showed
that every plane field is homotopic to a contact structure. The main
result of this paper is to extend this existence result to
5--manifolds.

\begin{thm}\label{main} 
On any closed oriented 5--manifold any almost contact structure
$(\eta,J)$ is homotopic, through almost contact structures, to a
contact structure $\xi$.
\end{thm}

There have been many partial results towards this theorem.  The first
breakthrough was due to Geiges \cite{Geiges91} who showed that on
simply connected 5--manifolds any almost contact structure was
homotopic to a contact structure. This result was reproven using open
book decompositions by van Koert in \cite{vanKoert08}. In the papers
\cite{GeigesThomas98} and \cite{GeigesThomas01} Geiges and Thomas
extended the existence results to some 5--manifolds with finite
fundamental group. In \cite{GeigesStipsicz10} Geiges and Stipsicz were
able to prove existence on some other 5--manifolds. The most recent
breakthrough is due to Casals, Pancholi, and Presas
\cite{CaselsPancholiPresas??} who proved that all possible first Chern
classes of almost contact structures can be realized by contact
structures on 5--manifolds. In particular, they established the above
theorem for manifolds without 2 torsion and gave a completely general
existence result (that is any almost contact 5--manifold admits some
contact structure). (Upon receiving a draft of this paper Casals, 
Pancholi, and Presas have informed the author that their arguments in 
\cite{CaselsPancholiPresas??} can be extended to obtain the main 
theorem above.)

The proof of Theorem~\ref{main} involves the use of an open book
decompositions of a manifold $M$. In Section~\ref{acs} we show that
every almost contact structure on a closed oriented 5--manifold $M$
can be ``supported'' by an open book decomposition. The open book
decomposition allows us to write $M$ as a neighborhood $N=Y\times D^2$
of an embedded 3--manifold and its complement $C$ that can be
expressed as a mapping torus of a 4--manifold $X$. We further break
$C$ into two pieces $C'$ and $C''$ where $C'=D^4\times S^1$ and $C''$
is the mapping torus of $\overline{X-D^4}$. We then use a result of
the author and Pancholi \cite{EtnyrePancholi11} to put a contact
structure on $C'$ that induces an overtwisted contact structure on the
boundary of $D^4\times \{\theta\}$ for each $\theta\in S^1$. We then
use work of Cieliebak and Eliashberg \cite{CieliebakEliashberg??} on
Weinstein manifolds to extend this contact structure over
$C''$. The final extension over $N$ relies on Eliashberg's
understanding of overtwisted contact structures on 3--manifolds in
\cite{Eliashberg89}.

We note that this proof is relatively elementary in that the two most
sophisticated tools that it uses are the stable s-cobordism theorem
(when quoting Quinn's result below on the existence of
open book decompositions \cite{Quinn79}) and a symplectic analog of 
the h-cobordism theorem due to Cieliebak and Eliashberg 
\cite{CieliebakEliashberg??}.

We also note that most of the steps in the proof have obvious analogs
in higher dimensions except the last step concerning the extension of
the contact structure over the neighborhood $N$ of the binding $Y$,
which relies heavily on Eliashberg's understanding of overtwisted
contact structures in dimension 3. So one might hope that a similar
approach would yield a general existence result in higher
dimensions. This is essentially a program laid out by Giroux
independently from the work here, though the author does thank Giroux for
comments that simplified the arguments in Section~\ref{acs}.

\begin{remark}
In the proof of Theorem~\ref{main} we make the somewhat surprising
observation that any 5--manifold $M$ has a fixed open book
decomposition that supports (in a weak sense, see
Section~\ref{sec:acsobd}) any almost contact structure on $M$ and more
generally any hyperplane field on $M$ can, in some sense, be thought
of as a perturbation of the pages of the open book. For a more precise
formulation of this fact see Theorem~\ref{getall}. This observation is
also true in higher dimensions but clearly far from true for open book
decompositions of 3--manifolds. The intuitive explanation for this
counterintuitive fact is that the pages of open books in higher
dimensional manifolds have more interesting topology than those of
3--dimensional open books.
\end{remark}

\smallskip
\noindent {\bf Acknowledgments:} The author thanks Kai Cieliebak and
Yasha Eliashberg for providing an advanced copy of
\cite{CieliebakEliashberg??} and especially thanks Cieliebak for
several very helpful discussions concerning Weinstein cobordisms.  He
additionally thanks David Gay, Emmanuel Giroux and Rob Kirby for
useful conversations and correspondence. The author is particularly
grateful to Dishant Pancholi who has made many helpful comments to
improve this paper and helped to identify a potential gap in the original
version of this paper.  This work was partially supported by NSF grant
DMS-0804820.

\section{Preliminary results and recollections}\label{pre}

In this section we first recall the notion of an open book
decomposition for a 5--manifold and discuss a special open book
decomposition on $S^5$. We we then discus of a generalization of Lutz
twists on 5--manifolds and finish by recalling Eliashberg's result on
the space of overtwisted contact structures.

\subsection{Open book decompositions of 5--manifolds}

An \dfn{open book decomposition} of a manifold $M$ is a pair $(Y,\pi)$
where $Y$ is a co-dimension 2 submanifold of $M$ that has a product
neighborhood $N=Y\times D^2$ in $M$ and $\pi\co (M-Y)\to S^1$ is a
fibration such that $\pi|_{N-Y}$ is the projection onto the
$\theta$--coordinate of $D^2$ where we give $D^2$ polar coordinates
$(r,\theta)$. We call $Y$ the \dfn{binding} of the open book
decomposition and $\overline{\pi^{-1}}(\theta)$ a \dfn{page} of the
open book.

There is another useful view of open books decompositions. If we are
given a pair $(X,\phi)$ where $X$ is a compact manifold with boundary
and $\phi\co X\to X$ is a diffeomorphism that agrees with the identity
map near $\partial X$, then we can construct a manifold as follows.
Let $T_\phi$ be the mapping torus of $\phi$, that is $T_\phi=X\times
[0,1]/\sim$ where $(x,1)\sim (\phi(x),0)$ for all $x\in X$. Notice
that $\partial T_\phi= Y\times S^1$ where $Y=\partial X$. We now glue
$Y\times D^2$ to $T_\phi$ so that the product structures on the
boundary agree. This gives a manifold $M_{(X,\phi)}$. We call
$(X,\phi)$ an open book decomposition for a manifold $M$ if $M$ is
diffeomorphic to $M_{(X,\phi)}$. Notice that the image of
$Y\times\{(0,0)\}$ in $M$ is the binding of an open book decomposition
as above and any open book decomposition as above can be constructed
from a pair $(X,\phi)$. We will use these notions interchangeably but
note that when using the second definition of an open book we must
always keep in mind a fixed diffeomorphism between $M$ and
$M_{(X,\phi)}$.

We have the following simple observation that we will use later.
\begin{lem}\label{obdcs} Given open book decompositions $(X_i,\phi_i)$
of $M_i$, $i=1,2$, define $X$ to be the boundary sum $X_1\natural X_2$
and $\phi$ to be the diffeomorphism of $X$ that restricts to $\phi_i$
on $X_i$. Then $M_{(X,\phi)}\cong M_{(X_1,\phi_1)}\#
M_{(X_2,\phi_2)}$.
\end{lem}

\subsection{An open book decomposition for $S^5$}

In this section we discusse a special open book decomposition of $S^5$
that we will need to ``stabilize'' open books of other 5--manifolds to
have nice properties that will be needed later.

\begin{thm}\label{s5ex}
There is an open books $(Y_\text{stab},\pi_\text{stab})$ for $S^5$
with page $X_\text{stab}$ the manifold $(S^2\times S^2)\#(S^2\times
S^2)$ with an open disk removed. Let $\phi_{stab}$ be the monodromy of
this open book. There is a Morse function $f_\text{stab}\co
X_\text{stab}\to\R$ for the page such that $f_\text{stab}$ and
$f_\text{stab}\circ \phi_\text{stab}$ can be connected by a family of
Morse functions having only 0 and 2--handles.
\end{thm}
\begin{proof}
In \cite{Saeki87}, Saeki showed there was an open book for $S^5$
with the given page. We sketch the argument here for the convenience
of the reader. Given diffeomorphisms $\phi$ of $X_\text{stab}$ that is
the identity on the boundary we see that the identity map $I$ on the
homology chains of $X_\text{stab}$ and $\phi_*$ induce the same map on
the homology chains in $\partial X_\text{stab}$. Thus $I-\phi_*$
induce a map from $H_2(X_\text{stab},\partial X_\text{stab})\to
H_2(X_\text{stab})$. Since $X_\text{stab}$ can be build with just 0
and 2--handles one may easily check that the open book $(X,\phi)$
gives a homotopy sphere if $I-\phi_*$ is an isomorphism. Using the
obvious product basis for $H_2(X_\text{stab})$ the matrix
\[
\Phi=
\begin{pmatrix}
0&0&1&0\\
0&1&0&1\\
-1&0&1&0\\
0&-1&0&0\\
\end{pmatrix}
\]
defines an automorphism of $H_2(X_\text{stab})$ that preserves the
intersection pairing. As $H_1(X_\text{stab})=0$ this determines a map
$I-\Phi\co H_2(X_\text{stab},\partial X_\text{stab})\to
H_2(X_\text{stab})$ that is easily seen to be an isomorphism. A result
of Wall \cite{Wall64} implies that there is a diffeomorphism
$\phi_\text{stab}$ of $X_\text{stab}$ that induces the $\Phi$. Thus
$(X_\text{stab},\phi_\text{stab})$ give an open book decomposition for
$S^5$ which we denote by $(Y_\text{stab},\pi_\text{stab})$.

One may easily show, {\em cf.\ }\cite{Kirby89}, that the
diffeomorphism given in Wall's theorem can be induced by a sequence of
2--handle slides. More specifically, there is a family of handle
decompositions $H_t$, $t\in[0,1]$, of manifold $X_t=X_\text{stab}$
that starts and ends with the same handle decomposition of
$X_\text{stab}=X_0=X_1$. In addition each $H_t$ only consists of 0 and
2--handles and the family describes isotopies of the attaching regions
and handles slides. We can use the handle decomposition of $X_t$ to
describe a diffeomorphism $\phi_t\co X_0\to X_t$ and $\phi_1\co
X_\text{stab}\to X_\text{stab}$ is a diffeomorphism that induces the
map $\Phi$ on $H_2(X_\text{stab})$, that is $\phi_1$ can be take to be
$\phi_\text{stab}$. Now if $f_\text{stab}$ is a Morse function on
$X_0=X_\text{stab}$ that corresponds to $H_0$ then set
$f_t=f_\text{stab}\circ \phi_t$. Thus we have a sequence of Morse
functions $f_t\co X_\text{stab}\to \R$ that correspond to 2--handle
slides and $f_1=f_\text{stab}\circ\phi_1$.
\end{proof}

\subsection{A notion of high dimensional Lutz twists}\label{sec:lutz}

Recall the standard contact structure $\xi_{std}$ on $S^3$ is obtained
as the set of complex tangencies to the unit $S^3$ in $\C^2$. There is
an overtwisted contact structure $\xi_{ot}$ on $\C^2$ in the same
homotopy class of plane field as $\xi_{std}$. This follows from
Eliashberg's classification of overtwisted contact structures
\cite{Eliashberg89}, or can easily be constructed by performing a full
Lutz twist, see Section~\ref{sec:ot} below, on a transverse unknot in
$\xi_{std}$.

\begin{prop}\label{twist}
There is a contact structure $\xi$ on $D^4\times S^1$ such that $\xi$
induces the overtwisted contact structure $\xi_{ot}$ on $\partial
D^4\times \{\theta\}$ for each $\theta\in S^1$. More precisely, if
$\alpha_{ot}$ is a 1--form for which $\xi_{ot}=\ker \alpha_{ot}$ then
in a neighborhood $S^3\times (1/2,1]\times S^1$ of the boundary of
$D^4\times S^1$ the contact structure is given by
\[ 
\xi=\ker (Kd\theta + t\alpha_{ot}),
\]
where $t$ is the coordinate on $(1/2,1]$, $\theta$ is the coordinate
on $S^1$, and $K$ is any positive constant.  Moreover, $\xi$ is
homotopic through almost contact structures to the almost contact
structure coming from the tangents to $D^4\times \{\theta\}$, for each
$\theta\in S^1$, where we think of $D^4$ as the unit disk in $\C^2$.
\end{prop}

This theorem follows easily from a results of Pancholi and the author
\cite{EtnyrePancholi11}. To state the result we first establish some
notation. Consider $T^2\times \R$ with coordinates $(\vartheta,
\varphi, r)$ and the contact structure
\[
\xi_{T^2\times \R} = \ker (\cos r \, d\vartheta + \sin r\, d\varphi).
\]
We will think of $\xi_{[a,b]}$ as the contact structure obtained from
$\xi_{T^2\times \R}$ by restricting it to $T^2\times [a,b]$. We also
denote $T^2\times \{r\}$ by $T_r$.

Notice that $T^2\times{[0,\pi/2]}$ has boundary $T_0\cup T_{\pi/2}$
and the characteristic foliation induced by $\xi_{[0,\pi/2]}$ is given
by a linear horizontal floatation on $T_0$ and linear vertical
foliation on $T_{\pi/2}$. So if we collapse the leave of the foliation
on the boundary of $T^2\times {[0,\pi/2]}$, or more precisely perform a
contact cut \cite{Lerman01}, the we obtain $S^3$ and $\xi_{[0,\pi/2]}$
induces the standard contact structure on $S^3$.  Similarly performing
a contact cut on the boundary of $T^2\times [0,5\pi/2]$ will also
produce $S^3$ but this time with the overtwisted contact structure
$\xi_{ot}$. (In particular it is easy to see that this contact
structure is obtained from the one above by a full Lutz twist.)

We now define a symplectic structure on a neighborhood of the boundary
of $Y\times [0,1]$, where $Y=T^2\times[0,1]$. Using an identification
of $[0,1]$ with $[0,\pi/2]$ we construct a diffeomorphism $\phi_0\co 
Y\times \{0\} \to T^2\times [0,\pi/2]$. Then on $\left(Y\times
[0,1/4]\right)\cup \left(T^2\times([0,1/4]\cup [3/4,1])\times
[0,1]\right)$ we let $\lambda$ be $t\, \phi_0^*\alpha$, where $t$ is
the coordinate on $[0,1]$. Similarly an identification of $[0,1]$ with
$[0,5\pi/2]$ can be used to build a diffeomorphism
$\phi_1\co Y\times\{1\}\to T^2\times[0,5\pi/2]$. Then on $Y\times[3/4,1]$
we can define $\lambda$ to be $t\, \phi_1^*\alpha$. If $\phi_1$ is
chosen correctly, $\lambda$ will be a well defined 1--form on a closed
neighborhood of the boundary of $Y\times [0,1]$ such that $d\lambda$
is a symplectic form with $Y\times \{0\}$ concave, $Y\times \{1\}$
convex and $(\partial Y)\times [0,1]$ flat. It is well known that
$\lambda$ cannot be extended over all of $Y\times[0,1]$ but if there
was an extension we could construct a contact from on $Y\times
[0,1]\times S^1$.  Pancholi and the author \cite{EtnyrePancholi11}
proved that this contact structure does exist.

\begin{lem}[Etnyre and Pancholi, 2011 \cite{EtnyrePancholi11}]\label{lem:twist}
With $Y$ and $\lambda$ as defined above, there is a contact structure
$\xi$ on $Y\times [0,1]\times S^1$ so that near the boundary
$\xi=\ker(Kd\theta+ \lambda)$, where $\theta$ is the coordinate on
$S^1$ and $K$ is any positive constant. Moreover, the contact
structure $\xi$ is homotopic through almost contact structures to
$\ker(Kd\theta+(t\, \phi_0^*\alpha))$.
\end{lem}
With this lemma in hand we return to the proof of our main theorem of
this section.
\begin{proof}[Proof of Proposition~\ref{twist}]
Simultaneously performing a contact cut on $\partial Y\times \{t\}$
for each $t$ gives a contact structure on $S^3\times [0,1]\times
S^1$. We can glue $D^4\times S^1$ with the contact structure $Kd\theta
+ \lambda$, where $\lambda$ is the standard Liouville form on $D^4$, 
to the lower boundary of $S^3\times [0,1]\times S^1$. This results in a
contact structure on $D^4\times S^1$ with the desired properties.
\end{proof}

\subsection{Plane fields and contact structures on 3--manifolds}\label{sec:ot}

Let $Y$ be a closed oriented 3--manifolds. A contact structure $\xi$
on $Y$ is call \dfn{overtwisted} if there is a disk $D$ embedded in
$Y$ that is tangent to $\xi$ along the boundary of $D$ and at one
point on the interior of $D$ and is transverse to $\xi$ elsewhere. The
disk $D$ is called an \dfn{overtwisted disk}.

Recall that given any contact structure $\xi$ on $Y$ we can alway
alter it to be overtwisted by a Lutz twist. To define a Lutz twist let
$K$ be a knot in $Y$ that is transverse to $\xi$. It is easy to show,
for example see \cite{Geiges08}, that $K$ has a standard neighborhood
$S$ and the boundary of $S$ has a neighborhood contactomorphic to
$(T^2\times[a,b],\xi_{[a,b]})$ for some $a<b$, where we are using the
notation from the previous section. Notice that one can replace the
contact structure on this neighborhood with
$(T^2\times[a,b],\xi_{[a,b+2\pi]})$ to obtain a new contact structure
on $Y$. We say that this contact structure is obtained from $\xi$ by a
full Lutz twist along $K$. It is well known, see for example
\cite{Geiges08} for a nice exposition of this, that the contact
structure obtained from $\xi$ by a full Lutz twist is homotopic as a
plane field to $\xi$; moreover, this homotopy can be done fixing the
plane fields along $K$ (which is called the core of the Lutz tube).
It is also simple to see that if full Lutz twists are performed
parametrically on a family of contact structure then the resulting
family is homotopic to the original family (also fixing the planes
along the cores of the Lutz tubes).

We fix a base point $p$ in $Y$ and a plane $P\subset TY$. Then denote
by $\mathcal{P}_p(Y)$ the space of plane fields on $Y$ that agree with
$P$ at $p$. Let $D$ be a disk in $Y$ that is tangent to $P$ at $p$
(and $p$ is on the interior of $D$). We let $\mathcal{C}_{ot}(Y)$ be
the space of contact structures on $Y$ for which $D$ and overtwisted
disk.
An amazing insight of Eliashberg relates these spaces. 
\begin{thm}[Eliashberg, 1989 \cite{Eliashberg89}]\label{otclass}
The natural inclusion map 
\[
i\co \mathcal{C}_{ot}(Y)\to \mathcal{P}_p(Y).
\]
is a homotopy equivalence. 
\end{thm}

\section{Almost contact structures}\label{acs}
In this section we study hyperplane fields and almost contact
structures on 5--manifolds.

\subsection{Hyperplane fields on 5--manifolds}\label{hypfields}

Let $M$ be an oriented 5--manifold. To study the homotopy classes of
hyperplane fields on $M$ we consider the bundle associated to the
tangent bundle $TM$ with fiber the Grassmann manifold of oriented
4--planes in $\R^5$:
\[
\xymatrix{\ar @{} [dr]
G^+(4,5)\ar[r] & E \ar[d] \\
       & M.       }
\]
Hyperplane fields on $M$ correspond to sections of this bundle. So we
are interested in studying the homotopy classes of sections of this
bundle. Throughout this section we will make no distinction between a
hyperplane field and a section of this bundle. Note that $G^+(4,5)$ is
diffeomorphic to $S^4$. (One may easily see this by noting that
oriented hyperplanes are in one-to-one correspondence with unit
vectors in $\R^5$, where the correspondence comes from an inner
product on $\R^5$.)

Studying the obstruction theory of this bundle we first observe that
$\pi_n(G^+(4,5))$ agrees with $[S^n, G^+(4,5)]$, the set of homotopy
classes of maps of $S^n$ to $G^+(4,5)$ (that is $G^+(4,5)$ is
$n$--simple). In addition because $TM$ is oriented it is well known
that the action of $\pi_1(M)$ on $\pi_n(G^+(4,5))$ coming from the
bundle $E$ is trivial. Thus the obstruction to homotoping two sections
of $E$ over the $n$--skeleton of $M$ lies in
$H^n\left(M;\pi_n(G^+(4,5))\right)$. Knowing that
\[
\pi_n(G^+(4,5))= \begin{cases}0 & n<4\\
\Z& n=4\\
\Z/2\Z& n=5
\end{cases}
\] 
it is clear that any section $s_1$ can be homotoped to $s_2$ on the
$3$--skeleton of $M$ and the obstruction to homotoping them together
on the 4--skeleton is an element
\[
d_4(s_1,s_2)\in H^4(M;\Z)
\]
and is a primary obstruction. Moreover if $s_1$ and $s_2$ agree on the
4--skeleton the, non-primary, obstruction to homotoping $s_1$ to $s_2$
over the 5--skeleton is an element
\[
d_5(s_1,s_2)\in H^5(M;\Z/2\Z).
\]

Given a plane field $\xi$ on a 3--manifold $Y$ and an open book
decomposition $(Y,\pi)$ of a 5--manifold $M$ we can construct a
hyperplane field as follows. Let $N=Y\times D^2$ be a neighborhood of
$Y$ in $M$, so that $\pi|_{(N-Y)}$ is just the projection to the
$\theta$--coordinate of $D^2$. On $N$ we take the hyperplane field
\[
\ker \left\{ f(r)d\theta + g(r) \alpha\right\}
\]
where $\alpha$ is a 1--form on $Y$ that defines $\xi$, $(r,\theta)$ are
polar coordinates on the unit disk $D^2$, $f$ is a non-decreasing
function equal to $r^2$ near 0 and constantly 1 near 1, and $g$ is a
non-increasing function equal to 1 near 0 and 0 near 1 for which $f$
and $g$ are never simultaneously zero. This hyperplane field can be
extended across $(M-N)$ by $\ker d\pi$. Denote this hyperplane field
$H(\xi)$. It is clear that $H$ gives a well defined map from the set
of plane fields on $Y$ to the set of 
hyperplane fields on $M$. Our main observation is that
under certain hypothesis this map is onto.
\begin{thm}\label{getall}
Let $M$ be any closed oriented 5--manifold and $(Y,\pi)$ any open book
decomposition of $M$. Then we have the following:
\begin{enumerate}
\item The map $H$ is well defined up to homotopy. That is if $\xi$ and
$\xi'$ are to homotopic plane fields on $Y$ then $H(\xi)$ is homotopic
to $H(\xi')$.
\item If the pages of $(Y,\pi)$ are handlebodies with handles of
index 0,1 and 2, then any homotopy class of hyperplane field is in
the image of the map $H$ defined above. More specifically, if $\eta$
is a hyperplane field on $M$ then there is some plane field $\xi$ on
$Y$ such that $H(\xi)$ is homotopic to $\eta$.
\end{enumerate}
\end{thm}
\begin{remark}
The obvious generalization of this theorem to higher dimensional
manifolds is also true and might be of use in trying to construct
contact structures homotopic to almost contact structures on higher
dimensional manifolds.
\end{remark}
\begin{proof}
The first statement is clear as we can apply the same construction
used in defining $H$ to the homotopy between $\xi$ and $\xi'$.

The second statement requires a bit more work. We begin by
reinterpreting the obstruction cohomology class $d_4$. Let $\eta$ and
$\eta'$ be two hyperplane fields on $M$. Thinking of $\eta$ and
$\eta'$ as sections of the bundle $G^+(4,5)\to E\to M$ we homotope
them so that they are transverse. Let
\[
\gamma(\eta,\eta')=\{x\in M | \eta(x)=\eta'(x)\}
\]
be the locus where $\eta$ and $\eta'$ agree.  
\begin{lem}\label{lem:recast}
With the notation above, the Poincar\'e dual of $d_4(\eta,\eta')$ is
$[\gamma(\eta,-\eta')]\in H_1(M;\Z)$.
\end{lem}
\begin{proof}
We begin by recalling how to compute $d_4(\eta,\eta')$. We can assume
that $\eta$ and $\eta'$ agree on the 3--skeleton of $M$. Now given a
4--cell $D$ of $M$ the sections $\eta$ and $\eta'$ each give a map from 
$D\to G^+(4,5)\cong S^4$ that agree along the boundary. We can thus 
use $\eta$ and $\eta'$ to define a map from $S^4\to S^4$ buy using 
$\eta$ on the upper hemisphere and $\eta'$ on the lower hemisphere. 
The degree of this map is the value of $d_4(\eta,\eta')$ on the 4--cell 
$D$.

We may isotope $\eta$ and $\eta'$ so that the 1--manifold
$\gamma(\eta,-\eta')$ is disjoint from the 3--skeleton of $M$ and
intersect the 4--skeleton transversely. It is clear that we may
compute the degree of the above mentioned map on $D$ as the
intersection number of $D$ and $\gamma(\eta,-\eta')$. (You can pick a
trivialization of $TM$ over $D$ so that $\eta$, say, is constant, then
$\gamma(\eta,-\eta')\cap D$ is just the preimage of some point on
$S^4$.)
\end{proof}

\begin{lem}
Let $M$ be a closed oriented 5--manifold with open book decomposition
$(Y,\pi)$ whose page is a handlebody with handle of index 0,1 and 2.
Fix a plane field $\xi$ on a $Y$.  As $\xi'$ ranges over all homotopy
classes of plane field on $Y$ the class
\[
d_4(H(\xi),H(\xi'))\in H^4(M,\Z)
\]
ranges over all of $H^4(M;\Z)$. \\
In particular, every hyperplane field $\eta$ on $M$ is homotopic, over
the 4--skeleton of $M$, to $H(\xi')$ for some plane field $\xi'$ on
$Y$.
\end{lem}
\begin{proof}
Using Lemma~\ref{lem:recast} we prove the first statement by showing
that $[\gamma(H(\xi),-H(\xi'))]$ ranges over all of $H_1(M;\Z)$ as
$\xi'$ ranges over all homotopy classes of plane field on $Y$. To this
end we first observe that the inclusion map
\[
i\co Y\to M
\]
induces a surjection
\[
i_*\co H_1(Y;\Z)\to H_1(M;\Z).
\]
To see this let $X$ be a page of the open book $(Y,\pi)$. A
neighborhood $V_0$ of $X$ can be identified with $X\times [0,1]$ and
$V_1=M\setminus V_0$ is also homeomorphic to $X\times [0,1]$. Notice
that $V_0$ and $V_1$ are handlebodies with handles of index less than
3. So $M$ is built from $V_0$ by adding handles of index greater than
2 (that is the handles in $V_1$ turned upside down). Thus the
inclusion of $V_0$ into $M$ induces an isomorphism on the first
homology $H_1(V_0;\Z)\to H_1(M;\Z)$. Of course the inclusion of $X$
into $V_0$ is a homotopy equivalence so $H_1(X;\Z)\cong
H_1(M;\Z)$. Finally notice that $X$ can be build from $Y=\partial X$
by attaching 4--dimensional handles of index 2, 3 and 4 (the handles in
$X$ turned upside down). Thus the inclusion map gives a surjection of
$H_1(Y,\Z)$ onto $H_1(X;\Z)$. The claim follows.

We now define a map $\widetilde H$ that is a slight variant of
$H$. Recall in the definition of $H$ we used the neighborhood
$N=Y\times D^2$ of $Y$ in $M$, where $D^2$ was the unit disk. We also
had functions $f(r)$ and $g(r)$ where $f(r)=r^2$ near 0 and is
constantly 1 near 1 and non-decreasing, and $g(r)=1$ near 0 and is 0
near 1 and non-increasing. We now specify that $g$ should be zero on
only the interval $[1/2,1]$. Now choose $\widetilde{g}(r)$ to be a
non-increasing function that is 1 near 0 and 0 near 1, but is 0 only
on $[3/4,1]$. For the moment let $\widetilde{f}=f$, but we will
perturb it later. Define $\widetilde H$ in the same manner as $H$
except use $\widetilde{f}$ and $\widetilde g$ in place of $f$ and
$g$. Notice that $H(\xi)$ and $\widetilde{H}(\xi)$ are homotopic as
hyperplane fields. Moreover, on $(Y\times \{(r,\theta)|
r\in[1/2,3/4)\})\subset M$, $H(\xi)$ and $\widetilde{H}(\xi')$ give
disjoint section. Thus $\widetilde{H}(\xi)$ can be perturbed, relative
to the region $(Y\times \{(r,\theta)| r\in[1/2,3/4)\})\subset M$, so
that it is transverse to $H(\xi)$ over the complement of $Y\times
\{(r,\theta)| r<1/2\}$. From now on we will denote by
$\widetilde{H}(\xi)$ this perturbed hyperplane field. Notice that
$H(\xi')$ for any other plane field $\xi'$ agrees with $H(\xi)$ on the
complement of $Y\times \{(r,\theta)| r<1/2\}$. Thus $H(\xi')$ is
transverse to $\widetilde{H}(\xi)$ on the complement of $Y\times
\{(r,\theta)| r<1/2\}$.

We claim with the appropriate choice of $\widetilde{f}$ and
$\widetilde{g}$, $H(\xi')$ will be transverse to $\widetilde{H}(\xi)$
on all of $M$ if $\xi$ and $\xi'$ are transverse on $Y$. To achieve
this we choose $\widetilde{f}$ to be zero only on $[0,1/2]$ (it is
still non-decreasing and 1 near 1). It is now a simple exercise to see
that $\widetilde{H}(\xi)$ and $H(\xi')$ are transverse on $Y\times
\{(r,\theta)| r<1/2\}$. (Though the exercise is somewhat easier if one
uses the dual picture to think of the hyperplane fields as unit vector
fields.)

Obstruction theory as discussed above shows $\xi$ and $\xi'$ have a
difference class $d_2(\xi,\xi')$ that obstructs homotoping them 
on the 2--skeleton of $Y$ and it is Poincar\'e dual to the
1--dimensional homology class given by the locus where they agree,
which we denote by $\gamma'$. From construction it is clear that
$\widetilde{H}(\xi)$ and $H(\xi')$ on $Y\times \{(r,\theta)| r<1/2\}$
only agree along $\gamma'$. So the Poincar\'e dual of
$d_4(\widetilde{H}(\xi),H(\xi'))$ is given by the homology class of
$\gamma'\cup \gamma''$ where $\gamma''$ is a 1--manifold in the
complement of $Y\times \{(r,\theta)| r<1/2\}$. Notice that $\gamma''$
is independent of $\xi'$. Moreover we know that by the appropriate
choice of $\xi'$ we can realize any homology class in $H_1(Y,\Z)$ as
the Poincar\'e dual of the difference class between $\xi'$ and
$\xi$. Thus we see that any element of $H_1(M;\Z)$ can be realized as
the Poincar\'e dual of the difference class
$d_4(\widetilde{H}(\xi),H(\xi'))=d_4(H(\xi),H(\xi'))$.

For the last statement recall that the difference class $d_4$ satisfies 
\[
d_4(\eta_1,\eta_3)=d_4(\eta_1,\eta_2)+d_4(\eta_2,\eta_3).
\]
Given a hyperplane field $\eta$ the first part of the proof guarantees
a plane field $\xi'$ such that
\[
d_4(H(\xi),H(\xi'))=-d_4(\eta,H(\xi)).
\] 
So we see that
$d_4(\eta,H(\xi'))=d_4(\eta,H(\xi))+d_4(H(\xi),H(\xi'))=0$. Thus
$H(\xi')$ is homotopic to $\eta$ over the 4--skeleton of $M$.
\end{proof}

Returning to the proof of Theorem~\ref{getall} we note that given any
hyperplane field $\eta$ on $M$ we can find some plane field $\xi$ on
$Y$ such that $H(\xi)$ is homotopic to $\eta$ on the 4--skeleton of
$M$. We are left to see that $\xi$ can be chosen so that the homotopy
can be extended over the 5--skeleton. Recall from above that if two
hyperplane fields agree on the 4--skeleton then the obstruction to
homotoping them on the 5--skeleton is a class in
$H^5(M;\Z/2\Z)=\Z/2\Z$.

We claim that if $\xi$ and $\xi'$ are homotopic over the 2--skeleton of
$Y$ and their obstruction to homotopy over the 3--skeleton is odd (this
obstruction is an element of $H^3(Y;\Z)=\Z$ and is well defined modulo
the divisibility of the Euler class of $\xi$), then $H(\xi)$ and
$H(\xi')$ are homotopic over the 4--skeleton of $M$ but
$d_5(H(\xi),H(\xi'))\not=0$. Thus if $\eta$ is homotopic to $H(\xi)$
over the 4--skeleton then it will be homotopic to either $H(\xi)$ or
$H(\xi')$ over all of $M$.

To verify the claim we begin by choosing representatives of $\xi$ and
$\xi'$ that agree except on the 3--handle of $Y$. Thus we know that
$H(\xi)$ and $H(\xi')$ agree on the complement of $N=Y\times D^2$. To
construct $M$ from the complement of $N$ we think of a handle
decomposition of $N$, coming from $Y$, turned upside down. Thus $M$ is
constructed from $\overline{M-N}$ by attaching a 2--handle, some 3 and
4--handles and a 5--handle. The 5--handle comes from the 3--handle of
$Y$ and thus we see that $H(\xi)$ and $H(\xi')$ agree on all of $M$
except the 5--handle $D^5$.

From the set up we know $D^5$ is the product of the 3--dimensional
3--handle $D^3$ with $D^2$. The plane fields $\xi$ and $\xi'$ give maps
from $D^3$ to $G^+(2,3)\cong S^2$ that agree on the boundary of
$D^3$. The obstruction to finding a homotopy of $\xi$ to $\xi'$ is an
element of $H^3(Y,\Z)$ that evaluates on $D^3$ to the ``Hopf
invariant" of the map $S^3\to G^+(2,3)\cong S^2$ given by $\xi$ on the
upper hemisphere of $S^3$ and $\xi'$ on the lower hemisphere. Denote
this map $h\co S^3\to S^2$.

Similarly the hyperplane fields $H(\xi)$ and $H(\xi')$ each give a map
from $D^5\cong D^3\times D^2$ to $G^+(4,5)\cong S^4$ and they agree on
$\partial D^4$. The obstruction to finding a homotopy of the
hyperplane field $H(\xi)$ to $H(\xi')$ is the homotopy class of the
map $H\co S^5\to S^4$ defined by $H(\xi)$ on the upper hemisphere and by
$H(\xi')$ on the lower hemisphere.

We claim the map $H\co S^5\to S^4$ factors through a map homotopic to
the map obtained form $h\co S^3\to S^2$ by a ``double
suspension". Recall that the suspension of a space $X$ is the quotient
of $X\times[0,1]$ by $X\times \{0\}$ and $X\times \{1\}$. Given a
function $f\co X\to X'$ there is clearly an induced suspension map
$S(f)\co S(X)\to S(X')$. It is well known, and easy to see, that
$S^n\cong S(S^{n-1})$. So we are claiming that our map $H$ factors
through a map $k\co S^5\to S^4$ that is homotopic to $S^2(h)$, in
other words there is a map $k'\co S^5\to S^5$ such that $H=k\circ
k'$. Moreover $k'_*\co \pi_5(S^5)\to \pi_5(S^5)$ is an
isomorphism. The Freudenthal suspension theorem implies that the
double suspension map from $\pi_3(S^2)\to \pi_5(S^4)$ is surjective
(here we have ignored base points, but due to the simple connectivity
of the spaces their inclusion would not alter the statement). Thus, as
$\pi_3(S^2)\cong \Z$ and $\pi_5(S^4)\cong\Z_2$, we see that $S^2(h)$
is not null homotopic and represents the generator of
$\pi_5(S^4)$. Thus $k$, which is homotopic to $S^2(h)$, induces a
surjective map $k_*\co \pi_5(S^5)\to \pi_5(S^4)$. As $k'_*$ gives an
isomorphism on $\pi_5(S^5)$ we see $H_*$ is surjective and hence $H\co
S^5\to S^4$ represents the generator of $\pi_5(S^4)$. Thus
$d_5(H(\xi),H(\xi'))$ is the mod 2 reduction of 3--dimensional
obstruction $d_3(\xi,\xi')$. This finishes the proof.

So we are left to justify the claim about $H$ and $S^2(h)$. We first
note that the suspension can be described in terms of
joins. Specifically, given a space $X$ its suspension is $S(X)\cong
X*S^0$. (Recall the join $X*Y$ of $X$ and $Y$ is obtained from
$X\times Y\times [0,1]$ where each of the sets $X\times \{y\}\times
\{0\}$ and $\{x\}\times Y\times \{1\}$ are identified to a point.) As
we know that $S^0*S^0=S^1$ it is clear that the double suspension can
be written $S^2(X)=X*S^1$. Now given a map $f\co X\to Y$ the double
suspension map $S^2(f)\co X*S^1\to Y*S^1$ is constructed as
follows. We begin by defining the map $X\times S^1\times [0,1]\to
Y\times S^1\times [0,1]$ by $(x,\theta, t)\mapsto (f(x),\theta,
t)$. Composition with the quotient map $Y\times S^1\times [0,1]\to
Y*S^1$ we get a map $X\times S^1\times [0,1]\to Y*S^1$. This map
clearly descends to a map $X*S^1\to Y*S^1$ which we define to be
$S^2(f)$.

In our case we start with the map $h\co S^3\to S^2$ and the map $H\co
S^5\to S^4$. These maps are defined with respect to some
trivialization of the tangent space $TY$ over the 3--handle $D^3$ and
of the tangent space $TM$ over $D^5$, respectively. Recall
$D^5=D^3\times D^2$ and over $D^5$ we have $TM|_{D_5}= TY_{D^3}\oplus
TD^2_{D^2}$.  Notice that the unit sphere $S^4$ in $\R^5$ is naturally
represented as the join of $S^2$ and $S^1$ where $S^2$ is the unit
sphere in the first 3 coordinates of $\R^5$ and the $S^1$ is the unit
sphere in the last 2 coordinates. Thus we me think of the
$S^4$--bundle in $TM|_{D^5}$ as a fiber wise join of the unit sphere
bundle in $TY|_{D^3}$ and the unit sphere bundle in $TD^2$.

Recall that the $S^5$ in the domain of $H$ is obtained by gluing two
copies of the 5--handle $D^5$ together. Moreover, $D^5=D^3\times D^2$
where $D^3$ is the 3--handle in $Y$ and $D^2$ is the fiber in the
neighborhood $N$ of $Y$. We glue the two copies of $D^3\times D^2$
together in two stages. We first glue along $(\partial D^3)\times
D^2$. This yields $S^3\times D^2$. Writing $D_1$ and $D_2$ for the two
copies of $D^3$ glued to form $D^3$ we see that $S^5$ is obtained from
$S^3\times D^2$ by gluing $D_1\times \{\theta\}$ to $D_2\times
\{\theta\}$ for $\theta\in \partial D^2$.  In addition we see that the
join $S^3*S^1$ is obtained from $S^3\times D^2$ (this is $S^3\times
S^1\times [0,1]/\sim$ where each $\{p\}\times S^1\times \{t\}$ is
collapsed to a point) by collapsing each $S^3\times \{\theta\}$ to a
point. Thus we clearly see that there is a map $k'\co S^5\to S^5$
obtained by thinking of the domain as two copies of $D^3\times D^2$
glued together and the image space as $S^3*S^1$. (Specifically the
first $S^5$ is obtained from $S^3\times D^2$ by partially collapsing
each $S^3\times \{\theta\}$ while the second is obtained by completely
collapsing the spheres.) This is clearly a degree 1 map and hence
induces an isomorphism on the homotopy groups of $S^5$. Below we will
denote $S^5$ created via the first method by $S^5_1$ and via the
second method as $S^5_2$. Clearly $S^2(h)\co S^5_2\to S^4$ and $H\co
S^5_1\to S^4$. We are now left to see that $H= k\circ k'$ where $k$ is
homotopic to $S^2(h)$.

Representing $S^4$ as $S^2*S^1$ from the splitting of $TM|_{D^5}$
above we can consider ``coordinates'' on $S^2*S^1$ to be
$(p,\theta,t)\in S^2\times S^1\times [0,1]$ as discussed above. Using
coordinates $(p,r,\theta)$ on $S^3\times D^2$, where $(r,\theta)$ are
polar coordinates on $D^2$, we see them map $H$ is given by
\[
(p,r,\theta)\mapsto\left(h(p),\theta, \frac 2\pi \tan^{-1}\frac{f(r)}{g(r)}\right).
\]
(To see this notice that we can map $S^2\times S^1\times [0,1]$ to the
unit sphere in $\R^5$ by the map $(p,\theta, t)\mapsto (p\cos
\frac{\pi}2 t, \theta \sin \frac\pi 2 t)$, where we think of $S^2$ is
the unit sphere in the first 3 coordinates of $\R^5$ and $S^1$ as the
unit sphere in the last 2 coordinates. Now the map $H$ from $S^3\times
D^2$ to $S^4\subset\R^5$ is given by $(p,r,\theta)\mapsto
(g(r)h(p),f(r)(\cos \theta, \sin \theta))$. So if $g$ and $f$ are
chosen so that $g^2+f^2=1$ we get the above representation of $H$ when
using the ``join coordinates" on $S^4$.)

The map $H\co S^5_1\to S^4$ descends to a map $k\co S^5_2\to S^4$ so
that $H=k\circ k'$. Now choosing a homotopy from $r\mapsto \frac 2\pi
\tan^{-1} \frac{f(r)}{g(r)}$ to the identity map $r\mapsto r$ we have an induced
homotopy $k\co S^5_2\to S^4$ to $S^2(h)\co S^5_2\to S^4$. Thus
completing the proof.
\end{proof}

\subsection{Almost contact structures}\label{sec:acs}

Recall an almost contact structure on an oriented 5--manifold $M$ is a
reduction of the structure group of $TM$ from $SO(5)$ to $U(2)\times
\mathbf{1}$. This is equivalent to a choice of homotopy class of
hyperplane field $\eta$ on $M$ together with a homotopy class of
complex structure $J$ on $\eta$. We will always think of an almost
contact structure as a pair $(\eta, J)$, up to homotopy.

To study the homotopy classes of almost contact structures on $M$ we
consider the bundle associated to the tangent bundle $TM$ with fiber
$SO(5)/U(2)$
\[
\xymatrix{\ar @{} [dr]
SO(5)/U(2)\ar[r] & E \ar[d] \\
       & M.       }
\]
Homotopy classes of sections of this bundle correspond to homotopy
classes of almost contact structures. Using obstruction theory to
understand the homotopy classes of sections of this bundle we need to
know the homotopy type of the fiber. To this end we recall that
$SO(5)/U(2)$ is diffeomorphic to $\C P^3$, see for example
\cite{Geiges08}. Thus all the relevant homotopy groups are zero except
$\pi_2(SO(5)/U(2))\cong \Z$. From this we see that the first and only
obstruction to homotoping one almost contact structure to another is
in $H^2(M;\Z)$. In particular, we have the following useful
observation.
\begin{lem}\label{lem:homotope} 
If two almost contact structures on the same 5--manifold $M$ are
homotopic over the 2--skeleton of $M$ then they are homotopic.
\end{lem}

\subsection{Almost contact structures and open book decompositions}\label{sec:acsobd}

Recall that in Section~\ref{hypfields} a map
\[
H_{(Y,\pi)}\co  \{\text{hyperplanes on $Y$}\}\to \{\text{hyperplanes on $M$}\}
\]
was assigned to an open book decomposition $(Y,\pi)$ of a 5--manifold
$M$. We say that an open book $(Y,\pi)$ \dfn{supports} an almost
contact structure $(\eta,J)$ if there is some hyperplane $\xi$ on the
binding $Y$ such that $\eta$ is (homotopic to) $H_{(Y,\pi)}(\xi)$ and
the plane field $\xi$ is a complex sub-bundle of $H_{(Y,\pi)}(\xi)$ on
$Y$ (that is $\xi$ on $Y$ is $J$ invariant).  Notice that this implies
that there is a homotopy of $(\eta,J)$ so that $\eta$ is transverse to
$Y$ and outside of a small neighborhood of $Y$ is tangent to the pages
of $(Y,\pi)$ and thus defines an almost complex structure on the
pages.

\begin{thm}\label{allacs}
Given an almost contact structure $(\eta, J)$ on a closed oriented
5--manifold, there is an open book decomposition $(Y,\pi)$ with pages
having a handle decomposition with handles of index less than or equal
to two that supports the given almost contact structure.

Moreover there is an overtwisted contact structure $\xi$ on $Y$ such
that $H(\xi)=\eta$ and one can assume that there is a neighborhood
$N=Y\times D^2$ of the binding $Y$ such that $\eta$ is tangent to the
pages of the open book outside of $N$ and at each point of $N$ the
plane $\xi$ is a $J$--complex sub-bundle of $H(\xi)$.
\end{thm}
To prove this proposition we will need a preliminary lemma.
\begin{lem}\label{nearbyJ}
If $J$ is a complex structure on the hyperplane field $\eta$ then
there is a $C^0$--neighborhood of $\eta$ in the space of hyperplane
fields such that $J$ induces a complex structure for all hyperplane
fields in this neighborhood. Moreover, the complex structures are
well-defined up to homotopy.
\end{lem}
\begin{proof}
Given $\eta$ choose any line field $L$ on $M$ that is transverse to
$\eta$. Let $\mathcal{O}$ be a $C^0$--neighborhood of $\eta$ in the
space of hyperplane fields that consists of hyperplanes transverse to
$L$. Note that $L$ and $\eta$ can be used to define a projection
$p\co TM\to \eta$. It is clear that at any point $x\in M$ the projection
$p$ maps $\eta'_x$ isomorphically onto $\eta_x$ for any
$\eta'\in\mathcal{O}$. Thus we can use $p$ to pull $J$ back form
$\eta$ to $\eta'$. The only choice made in this construction was the
choice of $L$, but clearly as we vary $L$ this only changes the
complex structures induced on $\eta'$ by a homotopy.
\end{proof}

\begin{proof}[Proof of Theorem~\ref{allacs}]
Quinn \cite{Quinn79} proved that there is an open book decomposition
$(Y,\pi)$ of $M$ whose pages have a handle decomposition with handles
only of index less than or equal to two.  Let $(\eta, J)$ be an almost
contact structure on $M$. By Proposition~\ref{getall} there is a
hyperplane field $\xi$ on $Y$ such that the induced hyperplane field
$H(\xi)$ is homotopic to $\eta$. Using Lemma~\ref{nearbyJ} we can pull
$J$ back along this homotopy so that $H(\xi)$ also has a complex
structure $J'$. Thus $(H(\xi),J')$ is an almost contact structure that
is homotopic to $(\eta,J)$. To simplify notation we now assume our
given almost contact structure is $(H(\xi),J)$. 

We are left to see that we can arrange that $\xi$ on $Y$ is
$J$--invariant. We will do this by altering $H(\xi)$ on the
neighborhood $N$ of the binding $Y$ to construct a new almost contact
structure and then show that this is homotopic to $(H(\xi),J)$ through
almost contact structures.

Recall $N=Y\times D^2$ where $D^2$ is the unit disk in $\R^2$.  We
write $D^2$ as the union $D_{1/4}\cup A_{1/4,1/2}\cup A_{1/2, 3/4}\cup
A_{3/4,1}$, where $D_{1/4}$ is the disk of radius $1/4$ and $A_{a,b}$
denotes the annulus $\{(r,\theta): a\leq r\leq b\}$. We can homotope
$(H(\xi),J)$ so that it is independent of the coordinates on $D^2$
near $Y\times \{(0,0)\}$. Now one can use a map from $A_{3/4,1}$ to
$D$ that is the identity near the outer boundary of the annulus, a
diffeomorphism away from the inner boundary and collapses the inner
boundary to the origin, to pull $(H(\xi),J)$ back to an almost contact
structure on $Y\times A_{3/4,1}$, call the resulting almost contact
structure $(\eta', J')$. Notice that $\eta'$ on $Y\times \{p\}$ for
any point on the inner boundary of $A_{3/4,1}$ is homotopic to
$TY\oplus \R$ (where we can take $\R$ to point in the radial direction
if we like). We can use this homotopy to extend $\eta'$ to a
hyperplane field on $A_{1/2,3/4}$ so that on the inner boundary of
this annulus the hyperplane field is $TY\oplus \R$. We can also use
the homotopy to extend $J'$ over this annulus so that it is
independent of the angular coordinate on the annulus.

To continue to define our almost contact structure we need to discuss
complex structures on $TY\oplus\R$.  Recall that complex structures on
a 4--dimensional bundle correspond to reductions of the structure
group from $SO(4)$ to $U(2)$ and thus correspond to sections of a
principal $SO(4)/U(2)$ bundle.  As $SO(4)/U(2)$ is homotopy equivalent
to $S^2$ and the bundle $TY\oplus \R$ is trivial, we see that complex
structures on this bundle correspond to maps $Y\to S^2$.  We can
concretely see this by noting that a unit vector in the $\R$ direction
of $TY\oplus \R$ will be mapped by the complex structure to a unit
vector in $TY$. The span of these two vectors is a complex line in
$TY\oplus \R$ and there is a unique, up to homotopy, complex structure
on the orthogonal complement of this complex line.  Thus a unit vector
field on $Y$, or dually an oriented plane field, determines the
complex structure on $TY\oplus \R$. Clearly there is some plane field,
say $\xi'$, on $Y$ that corresponds to $J'$ on $TY\oplus \R$ (thought
of as a hyperplane field along the inner boundary of $Y\times
A_{1/2,3/4}$). We can assume that $\xi'$ is an overtwisted contact
structure.

We can now define $\eta'$ on $D_{1/4}$ using the form $f(r)\, d\theta+
g(r)\, \alpha'$, where $\alpha'$ is a contact form for $\xi'$ and $f$
and $g$ are functions analogous to the ones used in the definition of
$H(\xi)$ in Section~\ref{hypfields}. With the appropriate choice of
$f$ and $g$ this is a contact form on the interior of $D_{1/4}$ and
the complex structure discussed in the previous paragraph thought of
as defined on $\eta'$ restricted to $Y\times \partial D_{1/4}$ extends
to a complex structure compatible the contact structure $\eta'$ on the
interior of $D_{1/4}$. In particular at any point in $Y\times D_{1/4}$
the plane field $\xi\subset TM$ can be assumed to be a $J'$--complex
line in $\eta'$. Finally by our choice of $J'$ and $\xi'$ above we can
extend $\eta'$ over $Y\times A_{1/4,1/2}$ to be $TY\oplus \R$ and
extend $J'$ by the homotopy from the paragraph above.

We finally note that $\eta'$ can be homotoped on $Y\times
(A_{1/2,3/4}\cup A_{3/4,1})$ so that it is $TY\times \R$ at every
point. Thus $\eta'$ is homotopic to $H(\xi')$ and the complex
structure $J'$ has the properties stated in the theorem. We are left
to see that $(H(\xi'),J')$ is homotopic to $(H(\xi),J)$. To this end
notice that if $U$ is a neighborhood of a page of the open book
$(Y,\pi)$ then $M$ has a handle decomposition with $U$ being the union
of 0, 1, and 2--handles and the other handles being index 3 and
above. Since $(H(\xi'),J')$ and $(H(\xi),J)$ agree along the cores of
the 0, 1 and 2--handles we see from Lemma~\ref{lem:homotope} that they
are homotopic on all of $M$. 
\end{proof}

\section{Cobordisms}

This section consists of two subsections. In the first we recall the
notions of Weinstein cobordisms and Cieliebak and Eliashberg's
``Weinstein flexibility results". The following subsection examines
Morse functions on the pages of open book decompositions of
5--manifolds.

\subsection{Weinstein cobordisms and flexibility}\label{sec:weinstein}

A \dfn{cobordism} $W$ is an compact $n$--manifold with boundary
$\partial W= -\partial_-W\cup \partial_+ W$. We say that $W$ is a
cobordism from $\partial_- W$ to $\partial_+W$.  A \dfn{Morse
cobordism} is a pair $(W,f)$ where $W$ is a cobordism and $f\co W\to
\R$ is a Morse function having $\partial_\pm W$ as regular level sets.

Recall that a 1--form $\lambda$ on a manifold $W$ is called a {\em
Liouville form} if $d\lambda$ is a symplectic form. Given such a
1--form there is a unique vector field $v$ such that $\iota_vd\lambda=
\lambda$. This vector field is called the {\em Liouville vector field}
associated to $\lambda$. Notice that if $\omega=d\lambda$ is a
symplectic structure on $W$ then $\lambda$ can be recovered from $v$
and $\omega$ by $\lambda=\iota_v \omega$. We call $(\omega, v)$ a
Liouville structure on $W$ if $d\iota_v \omega=\omega$.

A \dfn{Weinstein cobordism} is a tuple $(W, \omega, v, f)$ where $W$
is a compact cobordism, $(\omega,v)$ is a Liouville structure on $W$
so that $v$ points out of $W$ along $\partial_+W$ and into $W$ along
$\partial_-W$, and $f$ is a Morse function on $W$ that is constant on
each boundary component and for which $v$ is a gradient-like vector
field. Recall $v$ is gradient-like for $f$ if there is some constant
$\delta>0$ and some metric so that
\[
df(v)\geq \delta (|v|^2+|df|^2).
\]

It is well known, see for example \cite{CieliebakEliashberg??}, that
the Weinstein structure gives $W$ a handle decomposition with handles
of index less than or equal to half the dimension of $W$. We also note
that $\iota_v \omega|_{\partial_\pm W}$ is a contact form on
$\partial_\pm W$ and thus Weinstein cobordisms are cobordisms between
contact manifolds.

Cieliebak and Eliashberg define a \dfn{flexible} Weinstein cobordism
on a 4--manifold $W$ to be one with only 0 and 1--handles or one whose
lower boundary is an overtwisted contact 3--manifold and all the
2--handles are attached along Legendrian knots that have overtwisted
disks in their complement.

\begin{thm}[Cieliebak and Eliashberg 2012, \cite{CieliebakEliashberg??}]\label{wct}
Let $(W,f)$ be a 4--dimensional Morse cobordism such that $f$ has no
critical points of index larger than $2$. Let $\eta$ be a
non-degenerate 2--form on $W$ and $w$ a vector field near
$\partial_-W$ such that $(\eta,w,f)$ is a Weinstein cobordism
structure on some neighborhood $N$ of $\partial_-W$. Suppose that the
contact structure induced by $\iota_w\eta$ on $\partial_-W$ is
overtwisted. Then there is a flexible Weinstein cobordism structure
$(\omega,v,f)$ on $W$ such that
\begin{enumerate}
\item $(\omega,v)=(\eta,w)$ on some neighborhood $N'$ of $\partial_-W$, and
\item the non-degenerate 2--forms $\omega$ and $\eta$ are homotopic on
$W$ relative to $N'$.
\end{enumerate}
\end{thm}

We now discuss a slight refinement of the above theorem that follows
directly from the proof of that theorem, \cite{Cieliebak}. We first
establish some notation. A \dfn{relative cobordism} $W$ is a compact
$n$--manifold with boundary $\partial W=\partial_-W\cup \partial_+ W
\cup \partial_v W$ where $\partial_v W\cong -[0,1]\times \partial
(\partial_- W)\cong [0,1]\times \partial (\partial_+W)$. We say $W$ is
a relative cobordism from the manifold with boundary $\partial_-W$ to
the manifold with boundary $\partial_+W$. We say the \dfn{vertical
boundary} of $W$ is $\partial_vW$. (There is of course an analogous
notion of relative Morse cobordism.)

A \dfn{relative Weinstein cobordism} structure on a relative cobordism
$W$ is a triple $(\omega, v, f)$ as in the ordinary definition of
Weinstein cobordism and in addition, $f$ must have no critical points
near $\partial_v W$ and $v$ must be tangent to $\partial_vW$. As in
the case of Weinstein cobordisms a relative Weinstein cobordism is
flexible if it only has 0 and 1--handles or if on its lower boundary
the induced contact structure is overtwisted and all the 2--handles
are attached along Legendrian knots that have overtwisted disks in
their complement.

\begin{thm}[Cieliebak and Eliashberg 2012, \cite{Cieliebak, CieliebakEliashberg??}]\label{wct2}
Let $(W,f)$ be a 4--dimensional relative Morse cobordism such that $f$
has no critical points of index larger than $2$. Let $\eta$ be a
non-degenerate 2--form on $W$ and $w$ a vector field defined near
$\partial_-W\cup \partial_v W$ such that $(\eta,w,\phi)$ is a
Weinstein cobordism structure on some neighborhood $N$ of
$\partial_-W$ and on some neighborhood $N_v$ of $\partial_vW$. Suppose
that the contact structure induced by $\iota_w\eta$ on $\partial_-W$
is overtwisted. Then there is a flexible Weinstein cobordism structure
$(\omega,v,f)$ on $W$ such that
\begin{enumerate}
\item $(\omega,v)=(\eta,w)$ on some neighborhood $N'$ of $\partial_-W$
and $N'_v$ of $\partial_vW$ ($N'$ can be any subset of $N$ on which
$(\eta,w)$ gives a Weinstein cobordism structure and similarly for
$N'_v$ and $N_v$), and
\item the non-degenerate 2--forms $\omega$ and $\eta$ are homotopic on
$W$ relative to $N'\cup N'_v$.
\end{enumerate}
\end{thm}

A Weinstein homotopy on a cobordism $W$ is a smooth family of
Weinstein structures $(\omega_t, v_t, f_t), t\in [0,1],$ were we allow
the functions $f_t$ to have birth-death type degenerations (that is,
they form a generic family of functions), and $(\omega_t, v_t)$ is a
smooth family of Liouville structures on $W$. We now state a slight
strengthening of Cieliebak and Eliashberg's Weinstein flexibility
theorem from \cite{CieliebakEliashberg??}. Just as for
Theorem~\ref{wct2} above, the proof of this is the same as for the
non-relative version in \cite{CieliebakEliashberg??}.

\begin{thm}[Cieliebak and Eliashberg 2012, \cite{Cieliebak, CieliebakEliashberg??}]\label{cethm2}
Let $(\omega_0,v_0, f_0)$ and $(\omega_1, v_1, f_1)$ be two flexible
Weinstein structures on a 4--dimensional relative cobordism $W$ from an
overtwisted contact structure on $\partial_- W$ to some contact
structure on $\partial_+ W$.  Let
\begin{enumerate}
\item $f_t,t\in[0,1]$, be a Morse homotopy without critical points of
index greater than two, and
\item $\beta_t, t\in [0,1]$, be a homotopy of non-degenerate 2--forms
connecting $\omega_0$ and $\omega_1$ such that there are neighborhoods
$N$ of $\partial_- W$ and $N_v$ of $\partial_v W$ and vector fields
$w_t$ defined on $N\cup N_v$ connecting $v_0|_{N\cup N'}$ to
$v_1|_{N\cup N'}$ such that $(\beta_t, w_t, f_t)$ has the structure of
a relative Weinstein cobordism when restricted to $N$ and when
restricted to $N_v$.
\end{enumerate}
Then there are sub-neighborhoods $N'$ of $\partial_-W$ in $N$ and
$N'_v$ of $\partial_v W$ in $N_v$ and there is a homotopy $(\omega_t,
v_t, f_t), t\in [0,1],$ of Weinstein structures agreeing with
$(\beta_t, w_t, f_t)$ on $N'\cup N'_v$, such that the paths of
2--forms $\beta_t$ and $\omega_t$, for $t\in [0,1]$, are homotopic
relative to $N'\cup N'_v$ and the endpoints.
\end{thm}

\subsection{Open books and stabilization}

Let $(Y,\pi)$ be an open book decomposition of a closed oriented
5--manifold $M$. Let $N=Y\times D^2$ be a closed neighborhood of $Y$
in $M$. We can think of $\pi$ as a fibration of $M\setminus N$ (and
$\pi$ extends over $N-Y$ by projection to the $\theta$--coordinate of
$D^2$).

We denote the pages (that is fibers of $\pi\co (M\setminus N)\to S^1$)
of the open book by $X_\theta=\pi^{-1}(\theta)$.  Fixing some
$\theta_0\in S^1$, notice that
\[
(M\setminus N)\setminus X_{\theta_0}\cong X\times [0,1]
\]
where $X=X_{\theta_0}$. Moreover there is some diffeomorphism
$\phi\co X\to X$ that is the identity near $\partial X$ so that the
mapping torus
\[
T_\phi=X\times[0,1]/(x,1)\sim(\phi(x),0)
\]
is diffeomorphic to $M\setminus N$.  

Let $f$ is a Morse function for $X_{\theta_0}$. We denote $f$ on
$X\times \{0\}$ by $f_0$ and think of $\phi^*f$ is a Morse function on
$X\times \{1\}$, which we denote by $f_1$. 
\begin{thm}\label{stabilize}
Given an open book $(Y,\pi)$ for a closed oriented 5--manifold $M$ and
a Morse function $f\co X\to \R$ on a page of the open book with critical
points of index less than or equal to 2, there is another open book
$(Y',\pi')$ with monodromy $\phi'$ and Morse function $f'\co X'\to \R$ on
its pages so that, in the notation above, when $M\setminus N'$ is cut
open along the page $X'$ to obtain $X'\times [0,1]$ the Morse function
$f_0=f'$ on $X'\times \{0\}$ and $f_1=\phi'^*f'$ on $X'\times \{1\}$
can be extended to a family of functions $f_t\co (X'\times\{t\})\to \R$,
$t\in [0,1]$, satisfying the following:
\begin{enumerate}
\item $f_t$ is either a Morse function or has a birth/death degenerate
critical point,
\item no $f_t$ has critical points if index larger than 2,
\item each $f_t$ has a unique index 0 critical point,
\item the index 0 critical is constant in $t$ and it can be
connected to the boundary by a gradient like flow line that is also 
constant in $t$ and $\phi'$ fixes a
neighborhood of the index 0 critical point and flow line.
\end{enumerate}
\end{thm}
\begin{proof}
Let $f_0=f$ and $f_1=f\circ \phi$. Let $f_t$, $t\in[0,1]$, be a
generic family of functions connecting $f_0$ to $f_1$. It is well
known, see \cite{Cerf70}, that such a family has the following
structure. For all but finitely many $\{t_1,\ldots, t_k\}\subset(0,1)$
the functions $f$ are Morse functions whose critical points all have
distinct values. At each $t_i$ there are either precisely two critical
points with the same value or all critical points have distinct
values, but one of the critical points is a birth or death point.

The \dfn{Cerf graphic} for the family $\{f_t\}$ consists of the points
$(t,u)\in [0,1]\times \R$ where $u$ is a critical value of $f_t$. It
is now a standard argument using manipulations of the Cerf graphic,
see \cite{Cerf70, Kirby78}, that we can alter $f_t$ relative to $f_0$
and $f_1$ so that there are no births or deaths of index 0 or 4
critical points. Thus $f_t$ has a unique index 0 critical point for
all $t$. One can isotope $\phi$ and then the functions $\{f_t\}$,
relative to $f_0$ and $f_1$, so that the critical point and a 
gradient-like flow line is fixed for
all $t$; moreover, one can further isotope $\phi$ and then the
functions relative to $f_0$ and $f_1$ so that they all agree on a
neighborhood of the index 0 critical point and flow line.

We are left to show how to remove the index 3 critical points. We will
give a modification of Fenn and Rouke's extension \cite{FennRourke79}
of Kirby's arguments in \cite{Kirby78} to the non-simply connected
setting. We will be considering the family $X_t$ of cobordisms upside
down. That is we use the functions $\{-f_t\}$ so in order to eliminate
the index 3 critical points of the $\{f_t\}$ we will eliminate the
index 1 critical points of the $\{-f_t\}$.

We begin by identifying a product neighborhood $N=(\partial X)\times
[a,b]$ of $\partial X$  so that $\partial X= (\partial X)\times \{a\}$ and 
isotoping $\phi$ so that it is the identity map on $N$. We can also
assume that $f|_N\co (\partial X)\times [a,b]\to [a,b]$ is just the
projection map (and thus $-f_0$ and $-f_1$ will also be the projection
map when restricted to $N$). We will denote the copy of $N$ in $X_i$
by $N_i$, $i=0,1$.

Following Kirby we can modify $\{-f_t\}$ so that all the index 1
critical points occur as shown on the left of Figure~\ref{cerf}.
\begin{figure}[ht]
  \relabelbox \small {\centerline{\epsfbox{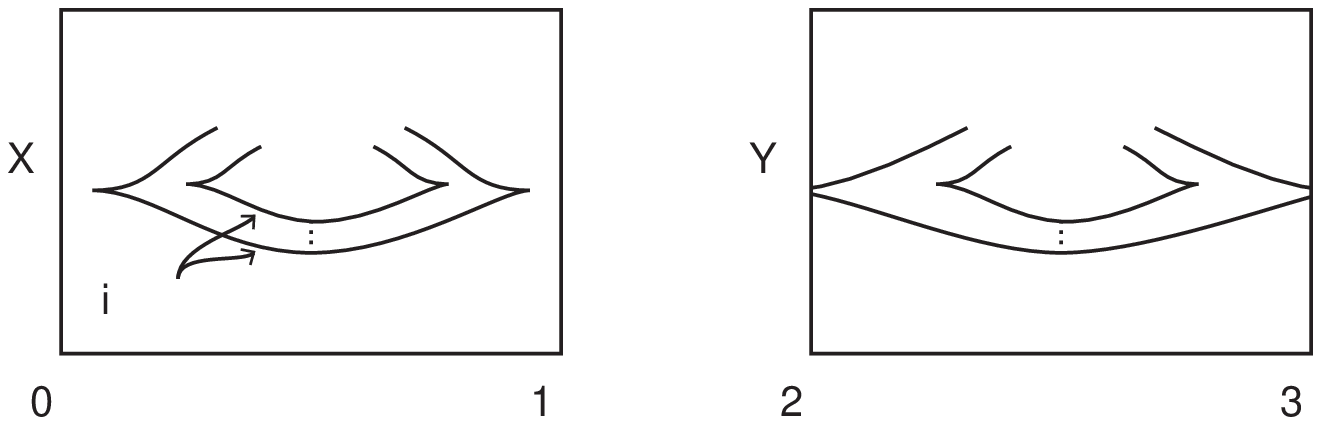}}} 
  \relabel{0}{$t=0$}
  \relabel {1}{$t=1$}
  \relabel{2}{$t=0$}
  \relabel {3}{$t=1$}
  \relabel{X}{$X$}
  \relabel{Y}{$X$}
  \relabel{i}{index 1 critical points}
  \endrelabelbox
        \caption{On the left, the birth and death of index 1 critical
points of $f_t$. On the right we have moved the birth death points for
the lowest index 1 critical point across the left and right boundary
of the Cerf diagram.}
        \label{cerf}
\end{figure}
Looking at the bottom most line of index 1 critical points we can
isotope $\{-f_t\}$ so that the birth and death points are pushed past
the edge of the Cerf graphic as shown on the right of
Figure~\ref{cerf}. We can assume that the canceling 1/2--handle pair in
$X_0$ and $X_1$ is contained a small 4--ball $B_0\subset N_0$ and
$B_1\subset N_1$, respectively, where $B_0=B_1$.

More precisely, we can assume that an index 1--handle, $h^1_i$ and
canceling index 2--handle, $h^2_i$ are attached to a 3--ball $\widetilde
B_i$ in $(\partial X)\times \{c\}\subset N_i$, for some $c\in(a,b)$
and $i=0,1$; and moreover $\widetilde{B}_0=\widetilde{B}_1$. The
attaching sphere for $h_0^2$ is a circle $A_0$ that is the union of
two arcs $a^0_1$ and $a^0_2$ with $a^0_1\subset \widetilde{B}_0$ and
$a^0_2$ contained in $h_0^1$.  Let $\gamma_0$ be an arc in
$\widetilde{B}_0$ that connects the two components of the attaching
region of $h^1_0$ and is parallel to $a_1$. There is an arc
$\gamma'_0$ in $h_0^1$ that is parallel to the core of $h_0^1$ and
completes $\gamma_0$ into a circle $C_0$. Notice that the 2--handle
$h_0^2$ shows that the circle $C_0$ is null-homotopic and in
particular bounds an embedded disk $D_0$.

Moving through the cobordisms $(X_t, -f_t)$ the attaching region of
the 1--handle changes by an isotopy in $(\partial X)\times \{c\}$
(since it is the bottom most 1--handle it does note slide over the
other 1--handles). This can be extended to an ambient isotopy of
$(\partial X)\times \{c\}$ and thus we may push $\gamma_0$ to a curve
$\gamma_1$ in $(\partial X)\times \{c\}\subset N_1$. We can further
extend this ambient isotopy to $X$ so that $D_0$ can be pushed to the
disk ${D}_1$ in $X_1$. In particular we get simple closed curves $C_t$
in $X_t$. Identifying $X$ with $X_0$ we let $X'$ be the result of
surgering $C_0$ in $X$. This will replace the 1--handle (that is index
1 critical point) with a 2--handle (that is index 2 critical point). We
can use the circles $C_t$ to surger all the $X_t$ simultaneously to
get $X'_t$ that fit together to form $X'\times[0,1]$. Moreover the
functions $-f'_t$ induced from the $-f_t$ and this surgery have one
less critical point of index 1. We claim that
\begin{enumerate}
\item $X'=X\#(S^2\times S^2)$ 
\item we may extend $\phi$ to a diffeomorphism $\phi'$ of $X'$ and
\item we may still assume that $f'_1=f'_0\circ \phi'$. 
\end{enumerate}
Once we establish these three claims we will see how to alter the pair
$(X',\phi')$ to an open book decomposition $(X'', \phi'')$ of $M$ with the 
desired properties.

We first observe that since $h_0^1$, $h_0^2$ and $\gamma_0$ are
attached in $B_0$ and so the surgery on $C_0$ with the appropriate
framing in $X_0$ clearly results in $X_0\# (S^2\times S^2)$.  Since we
can surger all the $C_t$ in $X_t$ (using the same framing) at the same
time and preserve the product structure (we could interpret this as
surgering the annulus that is formed as the union of the $C_t$ in
$X\times [0,1]$) we see that the surgered family of cobordisms
$X\times[0,1]$ becomes $X'\times [0,1]$ where $X'=X\# (S^2\times S^2)$
and thus establish Claim~(1). We note that we can extend the functions
$f_t$ restricted to the complement of the surgered regions to the
surgery regions to obtain new functions $f'_t$.

We now establish Claims~(2) and~(3) by analyzing the curve
$C_1$. Notice that $C_1$ runs over the 1--handle $h_1^1$ and no other
1--handles (again since we are considering the bottom most 1--handle
it does not slide over the other 1--handles). In particular
$C_1={\gamma}_1\cup \gamma_1'$, where $\gamma_1'$ is the analog of
$\gamma_0'$ that runs over the 1--handle $h_1^1$ in $X_1$.

The curve $C_1$ bounds the disk $D_1$ in $N_1=(\partial
X_1)\times[a,b]$ and so is null-homotopic.  The 2--handle $h_1^2$ is
attached along a curve $A_1=a_1^1\cup a_1^2$ where $a_1^1$ is in
$\widetilde{B}_1$ and $a_2^1$ is in the handle $h_1^1$. The arcs
$a_1^1$ and $\gamma_1$ can be connected in $(\partial X_1)\times
\{c\}$ to from a simple closed curve $c_1$. The 2--handle $h_1^2$ can
be used to show that $C_1$ and $c_1$ are isotopic and thus $c_1$
is null-homotopic in $(\partial X_1)\times \{c\}$ (since $(\partial
X_1)\times \{c\}$ is homotopy equivalent to $N_1$). We can use this
null-homotopy to isotope $\gamma_1$ in $(\partial X_1)\times \{c\}$ so
that it is parallel to $a_1^1$ (of course this isotopy may cross
$a_1^1$). The family of functions $-f_t$ for $t\in[0,1]$ can be
extended to a family for $t\in[0,2]$ that is constant for
$t\in[1,2]$. The isotopy above allows us to extend $C_t$, $t\in[0,1]$
to a family of curves $C_t$ in $X_t$ for $t\in[0,2]$ so that
$-f_2=-f_0\circ \phi$ and $C_0=\phi(C_2)$. Thus the $C_t$ form a
torus in the mapping torus of $\phi$ and we can perform surgery on
$C_t$ in $X_t$ simultaneously for all $t$. Reparameterizing the
interval $[0,2]$ establishes Claims~(2) and~(3).

Given $X_t'$, $-f_t'$ and $\phi'$ as above we construct $X''$ by
connect summing $X'$ with $S^2\times S^2$. More specifically we add
two 2--handles to $X'$ to obtain $X''=X'\# S^2\times S^2$. Since these
2--handles are attached to $\partial X'$ and $\phi'$ is the identity
map here we can clearly extend $\phi'$ over $X''$. Moreover we can
extend the $-f_t'$ over $X''$ to obtain Morse functions on $X_t''=X''$
that satisfy $f''_1=f_0''\circ \phi'$.

Recall that $X''=X\#Z$ where $Z=S^2\times S^2\# S^2\times S^2$. Using
Theorem~\ref{s5ex} we can perform handle slides on the $Z$ part of
$X''$, encoded by Morse functions $f''_t$, $t\in [1,2]$, to obtain a
diffeomorphism $\psi\co X''\to X''$ that is the identity away from $Z$
and gives the monodromy for $S^5$ described in the theorem. Thus
setting $\phi''=\phi'\circ \psi$ we have a diffeomorphism of $X''$ and
a family of Morse functions $f_t''$, $t\in [0,2]$, such that
$f''_2=f''_0\circ \phi''$. Moreover it is clear that $(X'',\phi'')$
describes the open book decomposition obtained from $(X,\phi)$ by
boundary summing with the open book from Theorem~\ref{s5ex}. That is
$(X'',\phi'')$ is an open book for $M\# S^5$ according to
Lemma~\ref{obdcs}. Moreover the Morse functions $f_t''$ have one less
3--handle than $f_t$. Thus continuing the above construction for each
3--handle in $f_t$ we eventually arrive at an open book and Morse
functions as described in the theorem.
\end{proof}

\section{Almost contact structures are homotopic to contact structures}
We are now ready to assemble the pieces discussed above to prove that
an almost contact structure on a 5--manifold can be homotoped to a
contact structure.

\begin{proof}[Proof of Theorem~\ref{main}]
Let $(\eta,J)$ be any almost contact structure on the closed oriented
5--manifold $M$.  Let $(Y,\pi)$ be an open book decomposition for $M$
with pages having no handles of index larger than
two. Theorem~\ref{stabilize} allows us to assume that there is a Morse
function $f$ on a page of $(Y,\pi)$ so that $f$ and $f\circ \phi$,
where $\phi$ is the monodromy of $(Y,\pi)$, can be connected by a
family of functions as described in the theorem.  By
Theorem~\ref{allacs} there is an overtwisted contact structure $\xi$
on $Y$ such that the hyperplane field $H(\xi)$ is homotopic to $\eta$
and using the homotopy we can think of $\eta$ as $H(\xi)$ and $J$ as a
complex structure on $H(\xi)$. Thus $J$ induces an almost complex
structures $J_\theta$ on the pages $X_\theta=\pi^{-1}(\theta), \,
\theta\in S^1,$ of $(Y,\pi)$.

We break $M$ into several parts. Let $N=Y\times D^2$ be a neighborhood
of the binding $Y$ in $M$, as in Theorem~\ref{allacs}. Let
$C=M\setminus N$ be the complement of the interior of $N$. We can
think of $\pi$ as a fibration $C\to S^1$ (that extends over $N-Y$ as
projection onto the $\theta$--coordinate of $D^2$). By a slight abuse
of notation we will refer to the fibers of $\pi\co C\to S^1$ as
$X_\theta$ and $J_\theta$ will be the almost complex structures on
$X_\theta$.  Letting $X=X_{\theta_0}$ for some fixed $\theta_0\in S^1$
we can write
\[
C=X\times [0,1]/(x,1)\sim(\phi(x), 0)
\]
for some diffeomorphism $\phi\co X\to X$ that is the identity near the
boundary of $X$. According to Theorem~\ref{stabilize} we can write $X$
as a union of $X'$ and $X''$ where $X'$ is a neighborhood of the fixed
index 0 critical point of $f\co X\to \R$ and $X''=X\setminus X'$ is the
complement of the interior of $X'$; moreover, $\phi$ is the identity
on $X'$. Thus we may decompose $C$ as
\[
C= C'\cup C'',
\]
where $C'=D^4\times S^1$ is the mapping torus of $\phi|_{X'}$ and
$C''$ is the mapping torus of $\phi|_{X''}$. 

Glancing at Steps~1 and~2 below will convince the reader that is is relatively 
easy to create a contact structure on $C$ (using Proposition~\ref{twist} for $C'$ and
Theorem~\ref{cethm2} for $C''$). However, it is not easy to extend this over
$N$. To achieve this extension we need to be more careful in Steps~1
and~2. For this reason we need to further decompose $X''$. More specifically, 
we will show below that there is a cylinder $W'=D^3\times [0,1]$ in
$X''$ that breaks $X''$ into two relative cobordisms $X''=W'\cup W$
where $W$ is a relative cobordism from $D^3$ to $Y\setminus B$ (where
$B$ is a ball in $Y$). Moreover $\phi$ is the identity on $X'$ and
$W'$, the complex structure $J$ restricted to $X'$ is standard and
restricted to $W'$ is compatible with (part of) the symplectization of
the standard contact structure on the 3--ball $D^3$. Moreover we have
a Morse function $f\co X''\to \R$ that satisfies the properties in
Theorem~\ref{stabilize} and has a gradient like vector field that is
non-zero on $W'$ and tangent to $\partial W'\cap \partial W$.

We begin by constructing a model situation. Let $B$ be the ball of
radius 1 in $\C^2$ with the standard complex structure. The complex
tangencies to $S^3=\partial B$ give the standard contact structure
$\xi_{std}$ on $S^3$. Let $D^3$ be a small ball in $S^3=\partial B$ on
which the contact structure is standard ({\em i.e.\ }a Darboux
ball). Gluing $D\times [1,2]$ to $B$ where $D\times\{1\}$ is
identified with $D\subset \partial B$ gives another ball $R$ and using
a portion of the symplectization of $(\xi_{std})_B$ we can extend the
standard symplectic structure and complex structure on $B$ to all of
$R$. Denote this complex structure $J'$

We now carefully construct the decomposition $C'\cup C''$ on
$C$. There is some neighborhood $U$ of $\partial X$ such that $\phi$
is the identity map on $U$ and hence each page $X_\theta$ has a
corresponding neighborhood $U_\theta$ that can be identified with
$U$. In particular, $U'=\cup_\theta U_\theta$ is a neighborhood of
$\partial C$ in $C$ and is diffeomorphic to $U\times S^1$.  We can
assume that $U\cong Y\times [0,2]$ and the construction in
Theorem~\ref{allacs} allows us to assume that under the
identifications of $U$ with $U_\theta$ the complex structure
$J_\theta$ is independent of $\theta$ on $U$. Again referring to the
proof of Theorem~\ref{allacs}, $Y\times\{t\}$ in each $U_\theta$ is
$J_\theta$ convex and the $J_\theta$ complex tangencies to
$Y\times\{t\}$ are $\xi$ (recall that $\xi$ is the overtwisted contact
structure on $Y$ such that $\eta=H(\xi)$).

Looking at the proof of Theorem~\ref{stabilize} we can assume the ball
$X'$ is a small neighborhood of a point in $U'$. Moreover we can take
an arc $\gamma$ from $\partial X'$ to $\partial X$ that is transverse
to all the $Y\times \{t\}$ in $U$.  Clearly there is a diffeomorphism
from $R$, constructed above, and $X'$ union a neighborhood of $\gamma$
so that $B$ maps to $X'$ and $D\times [1,2]$ maps to a neighborhood of
$\gamma$. We can also arrange that $D\times\{t\}$ maps to a ball in
$Y\times \{t\}$ and by (a parametric version of) Darboux's theorem the
$J'$--complex tangencies to $D\times\{t\}$ in $R$ map to the
$J$--complex tangencies to the ball in $Y\times\{t\}$. We can now
homotope, if necessary, the complex structure $J$ so that it agrees
with $J'$ on the image of $R$ and the complex tangencies to $Y\times
\{t\}$ are not changed for $t\in(1,2]$ (but we might change the
complex tangencies on $[0,1]$). Since $J$ is $\theta$--invariant in
$U'$ we can clearly arrange this on each page $X_\theta$.  We now set
$W=X'\setminus R$ and $W'=D\times [1,2]$ to get the desired
decomposition of $X''$ and notice that Theorem~\ref{stabilize} still
gives the desired Morse function.

\smallskip
\noindent
{\bf Step 1:} {\em The contact structure on $C'$.}
Proposition~\ref{twist} gives a contact structure $\zeta$ on
$C'=D^4\times S^1$ such that $\zeta$ induces the overtwisted contact
structure $\xi_{ot}$ on $\partial D^4\times \{\theta\}$ for each
$\theta\in S^1$.  The proposition also gives the following specific
form for the contact structure near $\partial C'$: if $\alpha_{ot}$ is
a 1--form for which $\xi_{ot}=\ker \alpha_{ot}$ then in a neighborhood
$S^3\times (1/2,1]\times S^1$ of the boundary of $D^4\times S^1$ the
contact structure is given by
\[
\zeta=\ker (Kd\theta + t\alpha_{ot}),
\]
where $t$ is the coordinate on $(1/2,1]$, $\theta$ is the coordinate
on $S^1$, and $K$ is any positive constant. Proposition~\ref{twist} 
allows us to arrange that $\ker(\alpha_{ot})$ is standard on the disk
$D^3$ where $W'$ intersects $\partial X'$. 

We note that $\xi_{ot}$ is homotopic to $\xi_{std}$ on $S^3$ relative
to the disk $D^3$ and thus there is a homotopy of the almost complex
structure on $(D^4=X')\cup X''$, relative to $W'$, so that $\zeta\cap
(\partial D^4\times \{\theta\})$ is the set of complex tangencies to
$\partial D^4=\partial_-X''$. By Proposition~\ref{twist}, $\zeta$ is
homotopic, through almost contact structures, to $\ker(Kd\theta+
\lambda)$, where $\lambda$ is a primitive for the standard symplectic
form on $D^4$. And this contact structure is homotopic to the almost
contact structure $\eta'$ coming from the tangencies to $B^4$ (with
the almost complex structure above).

\smallskip
\noindent
{\bf Step 2:} {\em Extending the contact structure to $C''$.} Recall that 
\[
C''= X''\times [0,1]/(x,1)\sim (\phi(x),0).
\]
In addition $X''=W'\cup W$ and we can think of $W'$ as part of the
symplectization of the standard contact structure on a 3--ball $D^3$,
that is $W'=D^3\times[1,2]$ and if $\alpha$ is a contact form for the
standard contact structure on $D^3$ then $d(t\alpha)$ is a symplectic
form on $W'$ and it is compatible with the almost complex structure
$J$. Thus we see that $W'\times S^1\subset C''$ already has a contact
structure given by $\ker(K\, d\theta+\alpha)$.  Note also that $(d
(t\alpha), v=\partial_t, f|_{W'})$ gives $W'$ the structure of a
relative Weinstein cobordism.

We are left to consider 
\[
W\times [0,1]/(x,1)\sim (\phi(x),0).
\]
We will denote $W\times\{t\}$ by $W_t$ and the almost complex
structure on $W_t$ by $J_t$.

Recall that $f$ and $f\circ \phi$ are connected by a family of
functions as in Theorem~\ref{stabilize}. Thinking of $f$ as a Morse
function on $W_0$ and $f\circ \phi$ as a Morse function on $W_1$ the
family of functions from Theorem~\ref{stabilize} can be thought of as
giving functions $f_t\co W_t\to \R$. (We are of course considering
the restriction of $f$ to $W$, but leave this out of the notation for 
convenience.) 

Recall that $J_t$ are all the same near the lower boundary $\partial_-
W_t$ so the complex tangencies on $S^3=\partial_-X_t$ give the
standard overtwisted contact structure $\xi_{ot}$. In addition all
the $J_t$ are the same near the vertical boundary $\partial_v W$. 
By extending $W$ into $W'$ and $X'$ slightly we can assume
that there is a neighborhood of $\partial_-W$ that has the structure of
a Weinstein cobordism and there is a neighborhood of $\partial_v W$
that has the structure of a Weinstein cobordism. 

Recall that for any almost complex structure the space of compatible
non-degenerate 2--forms is contractible. Thus we can use the almost
complex structures $J_t$ to find a path $\beta_t$ of non-degenerate
2--forms on $W_t$ such that $\beta_1=\phi^*\beta_0$.  Moreover we can
assume that $\beta_i$ agrees with the 2--forms defining the Weinstein
cobordism structures on the neighborhoods of $\partial_-W_t$ and
$\partial_v W_t$.

We now claim that there are vector fields $v_i$ on $W_i$, $i=0,1$ such
that $(\phi^{-1})_*v_1=v_0$ and we can assume that $(\beta_0, v_0,
f_0)$ and $(\beta_1,v_1,f_1)$ are flexible Weinstein structures on
$W_0$ and $W_1$, respectively.

To see this we first note that via homotopy we can assume that
$\beta_t$ and $f_t$ are fixed near $t=0$ and $t=1$. Now recall that
Theorem~\ref{wct2} says that there is a homotopy of $\beta_0$ to a
symplectic form $\omega$ and there exists a vector field $v$ such that
$(\omega, v, f_0)$ is a flexible relative Weinstein cobordism
structure on $W_0$.  We can now extend $W\times [0,1]$ to, say,
$W\times [-\epsilon, 1]$, for some $\epsilon>0$, and use the homotopy
of 2--forms between $\beta_0$ and $\omega$ to define non-degenerate
2--froms on $W_t$ for $t\in[-\epsilon, 0]$. We can also extend $f_t$
so that it is constant on $[-\epsilon, 0]$. Moreover pulling back the
homotopy of 2--forms via $\phi$ we can also extend our $\beta_t$ and
$f_t$ over $[1,1+\epsilon]$. We now have $(\beta_t,f_t)$ defined for
all $t\in [-\epsilon, 1+\epsilon]$ so that the structures at the
endpoints, together with $v$ and $\phi_*v$, define relative Weinstein
cobordism structures on $W$.  Moreover all the properties from
Theorem~\ref{stabilize} for the functions $f_t$ are unchanged and
$W\times[-\epsilon, 1+\epsilon]$ can be used to recover $C''$. Thus
re-parameterizing $[-\epsilon, 1+\epsilon]$ to $[0,1]$ establishes the
claim. (Notice that the homotopy of the 2--forms can be taken to be
fixed in a neighborhood of $\partial_-W\cup \partial_v W$.)

Noting that $f_t$, $t\in [0,1]$, is a Morse homotopy without critical
points of index greater than 2, we can now apply Cielieback and
Eliashberg's theorem, Theorem~\ref{cethm2} above, to get a path
$(\omega_t, v_t, f_t)$, $t\in[0,1]$, of Weinstein structures such that
$\omega_t$ is homotopic to $\beta_t$ relative to the end points and a
neighborhood of $\partial_-W$ and $\partial_v W$.  Now set
$\lambda_t=\iota_{v_t}\omega_t$ and consider the 1--form
\[
\alpha=Kdt + \lambda_t
\]
on $W$, where $K$ is some constant. Note
\[
\alpha\wedge d\alpha \wedge d\alpha= K\, dt\wedge d\lambda_t\wedge
d\lambda_t + 2\, dt\wedge \frac{d\lambda_t}{dt}\wedge \lambda_t\wedge
d\lambda_t.
\]
So for any $K$ sufficiently large, $\alpha$ defines a contact from on
$W\times [0,1]$ that extends the one on $W'\times [0,1]$ and clearly
descends to a contact from on $C''$.  Since we have a precise formula
for the contact form near $\partial_-X\times S^1\subset C''$ that
agrees with the one near $\partial C'\subset C'$ we can clearly extend
the contact form from Step 1 over $C''$ using $\alpha$. We denote the
associated contact structure $\zeta$. Moreover the contact structure
$\zeta$ is clearly homotopic through almost contact structures to the
tangents to the pages of the open book on $C''$ (just let the constant
$K$ go to infinity).

We wish to use Eliashberg's computation of the homotopy type of
overtwisted contact structures, Theorem~\ref{otclass}. To this end we
study the loop of contact structures $\xi_\theta=\ker(\lambda_\theta)$
on $Y$. First notice that the homotopy of 1--forms from the $\beta_t$,
compatible with the $J_t$, to the symplectic forms $\omega_t$, can be
covered by a homotopy of almost complex structures $(J_t)_s$,
$s\in[0,1]$. We can arrange
that the complex structures $(J_t)_1$ that are compatible with
$\omega_t$ also have $\xi_\theta$ as the set of complex tangencies to
$Y=\partial_+X''$.

The original almost complex structures on $X_\theta$ induced the
constant loop of plane fields $\xi$ on $Y$ as the complex tangencies
to $\partial_+X''_\theta=Y$. The homotopy of almost complex structures
on $(J_t)_s$ induces a homotopy of this loop to the loop
$\xi_\theta=\ker(\lambda_t)$. Moreover this loop is constant on the
ball $\partial_+W'\subset Y$. Thus $\xi_\theta$ is homotopic to a
constant loop of plane fields in the space $\mathcal{P}_p(Y)$ (for any
$p\in \partial_+W'$). (We are using the notation from
Section~\ref{sec:ot}.) To apply Theorem~\ref{otclass} we need to see
that we can assume that $\xi_\theta$ have a fixed overtwisted disk
through $p$.

\smallskip
\noindent
{\bf Step 3:} {\em Create a fixed overtwisted disk in each fiber of
$\partial C$.} We may identify a neighborhood $N''$ of $\partial C$ in
$C$ with $(1/2,1]\times Y\times S^1$. Under this identification the
contact form $\alpha$ can be written $Kd\theta + s \alpha_\theta$,
where $s$ is the coordinate on $(1/2,1]$, $\theta$ is the coordinate
on $S^1$ and $\alpha_\theta$ is a contact form on $Y$ for each
$\theta$. Moreover in $D^3=\partial_+W'\subset Y$ the contact forms
are independent of $\theta$ and induce the standard contact structure
on the ball. Let $S$ be a solid torus in $D^3$ that contains the point
$p$ discussed in the previous paragraph and let $T^2\times [0,1]$ be a
neighborhood of $\partial S$ that does not contain $p$. We use the
contact structure constructed in Lemma~\ref{lem:twist} to remove
$(1/2,1]\times T^2\times[0,1]\times S^1$ from $C$ and replace it with
the contact structure on $(1/2,1]\times T^2\times[0,1]\times S^1$
described in the lemma.  In particular we can assume that $S$ is a
``Lutz tube" and, more to the point, there is a disk $D$ in $D^3$ that
is overtwisted in the contact structures $\xi'_\theta$ induced on $Y$
thought of as the boundary of $X_\theta$ with this new contact
structure on $C$. Each $\xi'_\theta$ differs from $\xi_\theta$ by a
Lutz twist. And by the discussion in Section~\ref{sec:ot} we see that
the loop $\xi'_\theta$ is homotopic to $\xi_\theta$ in
$\mathcal{P}_p(Y)$. This completes Step 3. For convenience we rename
$\xi'_\theta$ to $\xi_\theta$.

\smallskip
\noindent
{\bf Step 4:} {\em Construct a contact structure on $N$.} From the
previous step we know that $\xi_\theta=\ker \alpha_\theta$ is a loop
of contact structures on $Y$ with a fixed overtwisted disk that is
contractible in $\mathcal{P}_p(Y)$. Thus according to Eliashberg's
result, Theorem~\ref{otclass}, the loop is contractible through
contact structures. That is for each $r\in [0,1]$ there is a loop
$\alpha_\theta^r$ of contact forms on $Y$ such that
$\alpha_\theta^1=\alpha_\theta$ and $\alpha_\theta^0$ is independent
of $\theta$. Moreover, we can assume that $\alpha_\theta^r$ is
independent of $r$ near 0 and near 1. Consider the 1--form
\[
\beta= f(r)\, d\theta + g(r)\, \alpha_\theta^{r} 
\]
on $Y\times D^2$. One can compute that
\[
\beta\wedge d\beta\wedge d\beta= g(f'g-g'f)\, dr\wedge
d\theta\wedge\alpha_\theta^{r}\wedge d\alpha_\theta^{r} - g^2\,
dr\wedge d\theta\wedge\frac{\partial\alpha_\theta^{r}}{\partial
r}\wedge \left(f\, d\alpha - g\,
\alpha\wedge\frac{\partial\alpha_\theta^{r}}{\partial \theta}\right).
\]
Where $\alpha_\theta^{r}$ is independent of $r$ the last term
vanishes. Thus as long as $g$ is positive and $(f'g-g'f)>0$, $\beta$
will be a contact from. Where $\alpha_\theta^{r}$ does depend on $r$
we will take $g$ to be constant and $f(r)=e^{cr}$. If $g$ is any fixed
constant and $c$ is chosen sufficiently large, $\beta$ will be a
contact form in this region too.

Specifically, let $0<r_1<r_2<1$ be constants such that
$\alpha_\theta^{r}$ is independent of $r$ outside $[r_1,r_2]$. Let
$r_3<r_4$ be any numbers in $(r_2, 1)$. We can choose $f(r)$ to be
$r^2$ near 0, $e^{cr}$ on $[r_1,r_2]$, strictly increasing on
$[0,r_4]$ and some large constant $K$ on $[r_4,1]$. We also choose
$g(r)$ to be 2 on $[0,r_3]$, strictly decreasing on $[r_3, 1]$ and
equal to $2-r$ near 1.

If $D_a$ is the disk of radius $a$ then we clearly see that $\beta$
defines a contact structure on $Y\times D_a$ for any $a>1$.

Using the coordinates $(1/2,1]\times Y\times S^1$ for a neighborhood
of $\partial C$ in $C$ from above, we can glue $Y\times D_a$ to $C$ by
a diffeomorphism
\[ 
\Psi\co  \{(p,r,\theta): 1<r<a\} \to  (1/2,1]\times Y\times S^1: (p,r,\theta) \mapsto (2-r,p, \theta).
\]
It is clear that $\Psi$ pulls the contact from from Step 2 back to the
one constructed in above.  Thus $\Psi$ can be used to glue the contact
structures on $N$ and $C$ together, thus extending $\zeta$ on $C$ over
$N$.

\smallskip
\noindent
{\bf Step 5:} {\em See that $\zeta$ is homotopic to $(\eta,J)$.}
Notice in the construction above $\zeta$ is clearly homotopic to
$(\eta,J)$ along a page of the open book $(Y,\pi)$. Thus, just as in
the proof of Theorem~\ref{allacs} we see that $(\eta,J)$ and $\zeta$
are homotopic as almost contact structures. (Recall this is simply
because almost contact structures on a 5--manifolds are determined up
to homotopy by their restriction to the 2--skeleton.)
\end{proof}

\begin{remark}
We note that in \cite{Giroux02} Giroux defines the notion of a contact
structure being \dfn{supported} by an open book. While we use open
books to construct our contact structures, none of the contact
structures coming from Theorem~\ref{main} are supported by the open
book used in the proof of the theorem. In particular, it is possible
that the given open book does not support any contact structures.
\end{remark}


\def\cprime{$'$} \def\cprime{$'$}

\end{document}